\documentclass[11pt,reqno]{amsart}
\usepackage{amsmath,amssymb}

 \makeatletter
 \evensidemargin  \oddsidemargin
 \makeatother

\begin{document}
\thispagestyle{empty}
 \setcounter{page}{1}

\begin{center}
{\large\bf Stability of $(\alpha,\beta,\gamma)-$derivations on Lie
$C^*-$algebras

\vskip.20in

{\bf M. Eshaghi Gordji and N. Ghobadipour  } \\[2mm]

{\footnotesize Department of Mathematics,
Semnan University,\\ P. O. Box 35195-363, Semnan, Iran\\
[-1mm] E-mail: {\ madjid.eshaghi@gmail.com
\\ghobadipour.n@gmail.com}}}
\end{center}
\vskip 5mm \noindent{\footnotesize{\bf Abstract.} Petr Novotn\'y and
Ji\v r\'l Hrivn\'ak \cite{Nov} investigated generalize the concept
of Lie derivations via certain complex parameters and obtained
various Lie and Jordan operator algebras as well as two one-
parametric sets of linear operators. Moreover, they established the
structure and properties of $(\alpha,\beta,\gamma)-$derivations of
Lie algebras.\\ We say a functional equation $(\xi)$ is stable if
any function $g$ satisfying the equation $(\xi)$ {\it approximately}
is near to true solution of $(\xi).$ In the present paper, we
investigate the stability of $(\alpha,\beta,\gamma)-$derivations on
Lie $C^*$-algebras associated with the following functional equation
$$f(\frac{x_2-x_1}{3})+f(\frac{x_1-3 x_3}{3})+
f(\frac{3x_1+3x_3-x_2}{3})=f(x_1).$$
 }

\vskip.10in
 %\footnotetext {Received March ???.  Revised September ????. }

  \newtheorem{df}{Definition}[section]
  \newtheorem{rk}[df]{Remark}
   \newtheorem{lem}[df]{Lemma}
   \newtheorem{thm}[df]{Theorem}
   \newtheorem{pro}[df]{Proposition}
   \newtheorem{cor}[df]{Corollary}
   \newtheorem{ex}[df]{Example}

 \setcounter{section}{0}
 \numberwithin{equation}{section}

\vskip .2in
\begin{center}
\section{Introduction}
\end{center}
The theory of finite dimensional complex Lie algebras is an
important part of Lie theory. It has several applications to
physics and connections to other parts of mathematics. With an
increasing amount of theory and applications concerning Lie
algebras of various dimensions, it is becoming necessary to
ascertain applicable tools for handling them. The miscellaneous
characteristics of Lie algebras constitute such tools and have
also found applications: Casimir operators \cite{Abl}, derived,
lower central and upper central sequences, Lie algebra of
derivations, radical, nilradical, ideals, subalgebras
\cite{Jac,Ran} and recently megaideals \cite{Pop}. These
characteristics are particularly crucial when considering possible
affinities among Lie algebras. Physically motivated relations
between two Lie algebras, namely contractions and deformations,
have been extensively studied, see e.g. \cite{Ger,Mon}. When
investigating these kinds of relations in dimensions higher than
five, one can encounter insurmountable difficulties. Firstly,
aside the semisimple ones, Lie algebras are completely classified
only up to dimension 5 and the nilpotent ones up to dimension 6.
In higher dimensions, only special types, such as rigid Lie
algebras \cite{Goz} or Lie algebras with fixed structure of
nilradical, are only classified \cite{Sno} (for detailed survey of
classification results in lower dimensions see e.g. \cite{Pop} and
references therein). Secondly, if all available characteristics of
two results of contraction/deformation are the same then one
cannot distinguish them at all. This often occurs when the result
of a contraction is one–parametric or
more–parametric class of Lie algebras.\\
We say a functional equation $(\xi)$ is stable if any function $g$
satisfying the equation $(\xi)$ {\it approximately} is near to true
solution of $(\xi).$  We say that a functional equation is
 {\it superstable} if every approximately solution is an exact solution
 of it \cite{TRa3}.\\
The stability problem of functional equations originated from a
question of Ulam \cite{Ul} in 1940, concerning the stability of
group homomorphisms. Let $(G_1,.)$ be a group and let $(G_2,*)$ be a
metric group with the metric $d(.,.).$ Given $\epsilon >0$, does
there exist a $\delta> 0$, such that if a mapping
$h:G_1\longrightarrow G_2$ satisfies the inequality
$d(h(x.y),h(x)*h(y)) <\delta$ for all $x,y\in G_1$, then there
exists a homomorphism $H:G_1\longrightarrow G_2$ with
$d(h(x),H(x))<\epsilon$ for all $x\in G_1?$ In the other words,
under what condition does there exist a homomorphism near an
approximate homomorphism? The concept of stability for functional
equation arises when we replace the functional equation by an
inequality which acts as a perturbation of the equation. In 1941, D.
H. Hyers \cite{Hy1} gave the first affirmative  answer to the
question of Ulam for Banach spaces. Let $f:{E}\longrightarrow{E'}$
be a mapping between Banach spaces such that
$$\|f(x+y)-f(x)-f(y)\|\leq \delta $$
for all $x,y\in E,$ and for some $\delta>0.$ Then there exists a
unique additive mapping $T:{E}\longrightarrow{E'}$ such that
$$\|f(x)-T(x)\|\leq \delta$$
for all $x\in E.$ Moreover if $f(tx)$ is continuous in
$t\in\mathbb{R}$  for each fixed $x\in E,$ then $T$ is linear.
Finally in 1978, Th. M. Rassias \cite{TRa1} proved the following
theorem.

\begin{thm}\label{t1} Let $f:{E}\longrightarrow{E'}$ be a mapping from
 a normed vector space ${E}$
into a Banach space ${E'}$ subject to the inequality
$$\|f(x+y)-f(x)-f(y)\|\leq \epsilon (\|x\|^p+\|y\|^p) \eqno \hspace {0.5
 cm} (1.1)$$
for all $x,y\in E,$ where $\epsilon$ and p are constants with
$\epsilon>0$ and $p<1.$ Then there exists a unique additive
mapping $T:{E}\longrightarrow{E'}$ such that
$$\|f(x)-T(x)\|\leq \frac{2\epsilon}{2-2^p}\|x\|^p  \eqno \hspace {0.5
 cm}(1.2)$$ for all $x\in E.$
If $p<0$ then inequality $(1.1)$ holds for all $x,y\neq 0$, and
$(1.2)$ for $x\neq 0.$ Also, if the function $t\mapsto f(tx)$ from
$\Bbb R$ into $E'$ is continuous in real $t$ for each fixed $x\in
E,$ then $T$ is linear.
\end{thm}
In 1991, Z. Gajda \cite{Gaj} answered the question for the case
$p>1$, which was raised by Rassias.  This new concept is known as
Hyers--Ulam--Rassias stability of functional equations. During the
last decades several stability problems of functional equations have
been investigated in the spirt of Hyers--Ulam--Rassias. See
\cite{Cz,Hy2,Jun,Rj3,Rj4,Rj5,Rj6,TRa1,TRa2,TRa3} for more detailed
information on stability of
functional equations.\\
Recently, the stability of various types of derivations has been
investigated by some mathematicians; see \cite{Bav,Par2,Par3,Par4,Par7,Par8}.\\
Petr Novotn\'y and Ji\v r\'l Hrivn\'ak \cite{Nov} investigated
generalize the concept of Lie derivations via certain complex
parameters and obtained various Lie and Jordan operator algebras
as well as two one- parametric sets of linear operators. Moreover,
they established the structure and properties of
$(\alpha,\beta,\gamma)-$derivations of Lie algebras.\\
A $C^*-$algebra $A,$ endowed with the Lie product $[x,y]=xy-yx$ on
$A,$ is called a Lie $C^*-$algebra (see
\cite{Par4,Par5,Par6}).\\
Let $A$ be a Lie $C^*-$algebra. A $\Bbb C-$linear mapping $d: A \to
A$ is a called a $(\alpha,\beta,\gamma)-$derivation of $A,$ if there
exist $\alpha,\beta,\gamma \in \Bbb C$ such that
$$\alpha d[x,y]=\beta [d(x),y]+\gamma [x,d(y)]$$ for all $x,y \in
A.$\\ In the present paper, we investigate the stability of
$(\alpha,\beta,\gamma)-$derivations on Lie $C^*-$algebras associated
with the following functional equation
$$f(\frac{x_2-x_1}{3})+f(\frac{x_1-3 x_3}{3})+
f(\frac{3x_1+3x_3-x_2}{3})=f(x_1).$$

\vskip 5mm

\section{Main results}
In this section, we investigate the superstability and the
Hyers--Ulam--Rassas stability of
$(\alpha,\beta,\gamma)-$derivations on Lie $C^*-$algebras.\\
Throughout this section, assume that $A$ is a Lie $C^*-$algebra with
norm $\|.\|_A.$ We need the following lemma in our main theorems.
\begin{lem}\label{t1}\cite{Par1}
Let $X$ and $Y$ be linear spaces and let $f: X \to Y$ be an additive
mapping such that $f(\mu x)=\mu f(x)$ for all $x \in X$ and all $\mu
\in \Bbb T^{1}:=\{\lambda \in \Bbb C ;~~~|\lambda|=1\}. $ Then the
mapping $f$ is $\Bbb C-$linear; i.e,
$$f(tx)=tf(x)$$ for all $x \in X$ and
all $t \in \Bbb C.$
\end{lem}
\begin{lem}\label{t2}
Let $f:A \to A$ be a mapping such that
$$\|f(\frac{x_2-x_1}{3})+f(\frac{x_1-3x_3}{3})+f(\frac{3x_1+3x_3-x_2}{3})\|_A \leq \|f(x_1)\|_A, \eqno(2.1)$$
for all $x_1,x_2,x_3 \in A.$ Then $f$ is additive.
\end{lem}
\begin{proof}
Letting $x_1=x_2=x_3=0$ in $(2.1),$ we get $$\|3f(0)\|_A \leq
\|f(0)\|_A.$$So $f(0)=0.$ Letting $x_1=x_2=0$ in $(2.1),$ we get
$$\|f(-x_3)+f(x_3)\|_A \leq \|f(0)\|_A=0$$ for all $x_3 \in A.$
Hence $f(-x_3)=-f(x_3)$ for all $x_3 \in A.$ Letting $x_1=0$ and
$x_2=6x_3$ in $(2.1),$ we get $$\|f(2x_3)-2f(x_3)\|_A \leq
\|f(0)\|_A=0$$ for all $x_3 \in A.$ Hence $$f(2x_3)=2f(x_3)$$ for
all $x_3 \in A.$ Letting $x_1=0$ and $x_2=9x_3$ in $(2.1),$ we get
$$\|f(3x_3)-f(x_3)-2f(x_3)\|_A \leq \|f(0)\|_A=0$$ for all $x_3 \in
A.$ Hence $$f(3x_3)=3f(x_3)$$ for all $x_3 \in A.$ Letting $x_1=0$
in $(2.1),$ we get
$$\|f(\frac{x_2}{3})+f(-x_3)+f(x_3-\frac{x_2}{3})\|_A \leq
\|f(0)\|_A=0$$ for all $x_2,x_3 \in A.$ So
$$f(\frac{x_2}{3})+f(-x_3)+f(x_3-\frac{x_2}{3})=0\eqno(2.2)$$ for all $x_2,x_3 \in
A.$ Let $t_1=x_3-\frac{x_2}{3}$ and $t_2=\frac{x_2}{3}$ in
$(2.2).$ Then $$f(t_2)-f(t_1+t_2)+f(t_1)=0$$ for all $t_1,t_2 \in
A$ and so $f$ is additive.
\end{proof}

Now we establish the superstability of
$(\alpha,\beta,\gamma)-$derivations as follows.
\begin{thm}\label{t2}
Let $p\neq 1$ and $\theta$ be nonnegative real numbers, and let
$f:A \to A$  be a mapping such that for some $\alpha,\beta,\gamma
\in \Bbb C$
$$\|f(\frac{x_2-x_1}{3})+f(\frac{x_1-3\mu x_3}{3})+\mu f(\frac{3x_1+3x_3-x_2}{3})\|_A \leq \|f(x_1)\|_A,\eqno(2.3)$$
$$\|\alpha f([x_1,x_2])-\beta[f(x_1),x_2]-\gamma [x_1,f(x_2)]\|_A \leq \theta (\|x_1\|_A^{2p}+\|x_2\|_A^{2p})\eqno(2.4)$$
for all $\mu \in \Bbb T^1:=\{\lambda \in \Bbb C~; |\lambda|=1\}$ and
all $x_1,x_2,x_3 \in A.$ Then the mapping $f:A \to A$ is a
$(\alpha,\beta,\gamma)-$derivation.
\end{thm}
\begin{proof}
Assume $p > 1.$\\
Let $\mu=1$ in $(2.3).$ By Lemma 2.2, the mapping $f:A \to A$ is
additive. Letting $x_1=x_2=0$ in $(2.3),$ we get $$\|f(-\mu x_3)+\mu
f(x_3)\|_B \leq \|f(0)\|_B=0$$ for all $x_3\in A$ and $\mu \in \Bbb
T^1.$ So $$-f(\mu x_3)+\mu f(x_3)=f(-\mu x_3)+\mu f(x_3)=0$$ for all
$x_3\in A$ and all $\mu \in \Bbb T^1.$ Hence $f(\mu x_3)=\mu f(x_3)$
for all $x_3 \in A$ and all $\mu \in \Bbb T^1.$ By Lemma $2.1,$ the
mapping $f:A \to A$ is $\Bbb C-$linear. Since $f$ is additive, then
It follows from $(2.4)$ that
\begin{align*}
&\|\alpha f([x_1,x_2])-\beta[f(x_1),x_2]-\gamma [x_1,f(x_2)]\|_A\\
&=\lim_{n \to \infty} 4^n
\|\alpha f(\frac{[x_1,x_2]}{2^n.2^n})-\beta[f(\frac{x_1}{2^n}),\frac{x_2}{2^n}]-\gamma [\frac{x_1}{2^n},f(\frac{x_2}{2^n})]\|_A\\
&\leq \lim_{n \to \infty}\frac{4^n
\theta}{4^{np}}(\|x_1\|_A^{2p}+\|x_2\|_A^{2p})\\
&=0
\end{align*}
for all $x_1,x_2 \in A.$ Thus for some $\alpha, \beta,\gamma \in
\Bbb C$
$$\alpha f([x_1,x_2])=\beta[f(x_1),x_2]+\gamma[x_1,f(x_2)]$$ for all $x_1,x_2 \in
A.$ Hence the mapping $f:A \to A$ is a
$(\alpha,\beta,\gamma)-$derivation. Similarly, one obtains the
result for the case $p<1.$
\end{proof}
Now, we establish  the Hyers--Ulam--Rassias
 Stability of $(\alpha,\beta,\gamma)-$derivations on Lie $C^*-$algebras.
\begin{thm}
Let $p>1$ and $\theta$ be nonnegative real numbers, and let $f:A
\to A$ with $f(0)=0$ be a mapping such that for some
$\alpha,\beta,\gamma \in \Bbb C$
\begin{align*}
&\|f(\frac{x_2-x_1}{3})+f(\frac{x_1-3\mu x_3}{3})+\mu
f(\frac{3x_1+3x_3-x_2}{3})-f(x_1)\|_A \\
&\leq \theta(\|x_1\|_A^p+\|x_2\|_A^p+\|x_3\|_A^p)\hspace{7
cm}(2.5)
\end{align*}
and
$$\|\alpha f([x_1,x_2])-\beta[f(x_1),x_2]-\gamma[x_1,f(x_2)]\|_A\\
\leq \theta (\|x_1\|_A^{2p}+\|x_2\|_A^{2p})\eqno(2.6)$$ for all $\mu
\in \Bbb T^1$ and all $x_1,x_2,x_3 \in A.$ Then there exists a
unique $(\alpha,\beta,\gamma)-$derivation $d:A \to A$ such that
$$\|d(x_1)-f(x_1)\|_A \leq \frac{\theta(1+2^p)\|x_1\|_A^p}{1-3^{1-p}}\eqno(2.7)$$
for all $x_1 \in A.$
\end{thm}
\begin{proof}
Let us assume $\mu=1,$ $x_2=2x_1$ and $x_3=0$ in $(2.5).$ Then we
get
$$\|3f(\frac{x_1}{3})-f(x_1)\|_A \leq
\theta(1+2^p)\|x_1\|_A^p\eqno(2.8)$$ for all $x_1 \in A.$ So by
induction, we have
$$\|3^nf(\frac{x_1}{3^n})-f(x_1)\|_A \leq \theta(1+2^p)\|x_1\|_A^p \sum_{i=0}^{n-1}3^{i(1-p)}\eqno(2.9)$$
for all $x_1 \in A.$ Hence
\begin{align*}
\|3^{n+m}f(\frac{x_1}{3^{n+m}})-3^mf(\frac{x_1}{3^m})\|_A
&\leq \theta (1+2^p)\|x_1\|_A^p \sum_{i=0}^{n-1}3^{(i+m)(1-p)}\\
&\leq \theta (1+2^p)
\|x_1\|_A^p\sum_{i=m}^{n+m-1}3^{i(1-p)}\hspace{1.7 cm} (2.10)
\end{align*}
for all nonnegative integer $m$ and all $x_1 \in A.$ From this it
follows that the sequence $\{ 3^n f(\frac{x_1}{3^n})\}$ is a
Cauchy sequence for all $x_1 \in A.$ Since $A$ is complete, the
sequence $\{3^n f(\frac{x_1}{3^n})\}$ converges. Thus one can
define the mapping $d:A \to A$ by $$d(x_1):=\lim_{n \to \infty}
3^n f(\frac{x_1}{3^n})$$ for all $x_1 \in A.$ Moreover, letting
$m=0$ and passing the limit $n \to \infty$ in $(2.10),$ we get
$(2.7).$ It follows from $(2.5)$ that
\begin{align*}
&\|d(\frac{x_2-x_1}{3})+d(\frac{x_1-3\mu x_3}{3})+\mu
d(\frac{3x_1+3x_3-x_2}{3})-d(x_1)\|_A\\
&=\lim_{n \to \infty}
3^n\|f(\frac{x_2-x_1}{3^{n+1}})+f(\frac{x_1-3\mu
x_3}{3^{n+1}})+\mu f(\frac{3x_1+3x_3-x_2}{3^{n+1}})-f(\frac{x_1}{3^n})\|_A\\
&\leq \lim_{n \to
\infty}\frac{3^n\theta}{3^{np}}(\|x_1\|_A^p+\|x_2\|_A^p+\|x_3\|_A^p)=0
\end{align*}
for all $\mu \in \Bbb T^1$ and all $x_1,x_2,x_3 \in A.$ So
$$d(\frac{x_2-x_1}{3})+d(\frac{x_1-3\mu x_3}{3})+\mu
d(\frac{3x_1+3x_3-x_2}{3})=d(x_1)$$ for all $\mu \in \Bbb T^1$ and
all $x_1,x_2,x_3 \in A.$ By Lemma $2.2,$ the mapping
$d:A \to A$ is additive. Hence by Lemma $2.1,$ $d$ is $\Bbb C-$linear.\\
It follows from $(2.6)$ that
\begin{align*}
&\|\alpha d([x_1,x_2])-\beta[d(x_1),x_2]-\gamma[x_1,d(x_2)]\|_A\\
&=\lim_{n \to \infty}9^n\|\alpha
f\frac{[x_1,x_2]}{3^n.3^n}-\beta[f(\frac{x_1}{3^n}),\frac{x_2}{3^n}]-\gamma[\frac{x_1}{3^n},f(\frac{x_2}{3^n})]
\|_A\\
&\leq \lim_{n \to
\infty}\frac{9^n\theta}{9^{np}}(\|x_1\|_A^{2p}+\|x_2\|_A^{2p})=0
\end{align*}
for all $x_1,x_2 \in A.$ So for some $\alpha,\beta,\gamma \in A$
$$\alpha d([x_1,x_2])=\beta[d(x_1),x_2]+\gamma[x_1,d(x_2)]$$
for all $x_1,x_2 \in A.$\\
So, $d$ is a $(\alpha,\beta,\gamma)-$derivation. Now, let $d^{'}:A
\to A$ be another  $(\alpha,\beta,\gamma)-$derivation satisfying
$(2.7).$ Then we have
\begin{align*}
&\|d(x_1)-d^{'}(x_1)\|_A=3^n
\|d(\frac{x_1}{3^n})-d^{'}(\frac{x_1}{3^n})\|_A\\
&\leq 3^n
(\|d(\frac{x_1}{3^n})-f(\frac{x_1}{3^n})\|_A+\|d^{'}(\frac{x_1}{3^n})-f(\frac{x_1}{3^n})\|_A)\\
&\leq \frac{2.3^n \theta(1+2^p)}{3^{np}(1-3^{1-p})}\|x\|_A^p~,
\end{align*}
which tends to zero as $n \to \infty$ for all $x_1 \in A.$ So we
can conclude that $d(x_1)=d^{'}(x_1)$ for all $x_1 \in A.$ This
proves the uniqueness of $d.$\\
Thus the mapping $d:A \to A$ is a unique
$(\alpha,\beta,\gamma)-$derivation satisfying $(2.7).$
\end{proof}
\begin{thm}
Let $p<1$ and $\theta$ be nonnegative real numbers, and let $f:A \to
A$ with $f(0)=0$ be a mapping  satisfying $(2.5)$ and $(2.6).$ Then
there exists a unique $(\alpha,\beta,\gamma)-$derivation $d:A \to A$
such that
$$\|d(x_1)-f(x_1)\|_A \leq \frac{\theta(1+2^p)\|x_1\|_A^p}{3^{1-p}-1}$$
for all $x_1 \in A.$
\end{thm}
\begin{proof}
Let us assume $\mu=1,$ $x_2=2x_1$ and $x_3=0$ in $(2.5).$ Then we
get
$$\|3f(\frac{x_1}{3})-f(x_1)\|_A \leq
\theta(1+2^p)\|x_1\|_A^p\eqno(2.11)$$ for all $x_1 \in A.$
Replacing $x_1$ by $3x_1$ in $(2.11)$ we have
$$\|3f(x_1)-f(3x_1)\|_A \leq
\theta(3^p)(1+2^p)\|x_1\|_A^p\eqno(2.12)$$ for all $x_1 \in A.$ It
follows from $(2.12)$ that $$\|f(x_1)-3^{-1}f(3x_1)\|_A \leq
\theta(3^{p-1})(1+2^p)\|x_1\|_A^p\eqno(2.13)$$ for all $x_1 \in
A.$ So by induction, we have
$$\|f(x_1)-3^{-n}f(3^nx_1)\|_A \leq \theta(1+2^p)\|x_1\|_A^p \sum_{i=1}^{n}3^{i(p-1)}\eqno(2.14)$$
for all $x_1 \in A.$ Hence
\begin{align*}
\|3^{-m}f(3^mx_1)-3^{-(n+m)}f(3^{n+m})\|_A
&\leq \theta (1+2^p)\|x_1\|_A^p \sum_{i=1}^{n}3^{(i+m)(p-1)}\\
&\leq \theta (1+2^p)
\|x_1\|_A^p\sum_{i=m+1}^{n+m}3^{i(p-1)}\hspace{1. cm} (2.15)
\end{align*}
for all nonnegative integer $m$ and all $x_1 \in A.$ From this it
follows that the sequence $\{ 3^{-n} f(3^nx_1)\}$ is a Cauchy
sequence for all $x_1 \in A.$ Since $A$ is complete, the sequence
$\{3^{-n} f(3^nx_1)\}$ converges. Thus one can define the mapping
$d:A \to A$ by $$d(x_1):=\lim_{n \to \infty} 3^{-n} f(3^nx_1)$$
for all $x_1 \in A.$ Moreover, letting $m=0$ and passing the limit
$n \to \infty$ in $(2.15),$ we get $$\|d(x_1)-f(x_1)\|_A \leq
\frac{\theta(1+2^p)\|x_1\|_A^p}{3^{1-p}-1}$$ for all $x_1 \in A.$
It follows from $(2.5)$ that
\begin{align*}
&\|d(\frac{x_2-x_1}{3})+d(\frac{x_1-3\mu x_3}{3})+\mu
d(\frac{3x_1+3x_3-x_2}{3})-d(x_1)\|_A\\
&=\lim_{n \to \infty}
3^{-n}\|f(\frac{3^n(x_2-x_1)}{3})+f(\frac{3^n(x_1-3\mu
x_3)}{3})\\
&+\mu f(\frac{3^n(3x_1+3x_3-x_2)}{3})-f(\frac{3^n(x_1)}{3})\|_A\\
&\leq \lim_{n \to \infty}\theta.
3^{n(p-1)}(\|x_1\|_A^p+\|x_2\|_A^p+\|x_3\|_A^p)=0
\end{align*}
for all $\mu \in \Bbb T^1$ and all $x_1,x_2,x_3 \in A.$ So
$$d(\frac{x_2-x_1}{3})+d(\frac{x_1-3\mu x_3}{3})+\mu
d(\frac{3x_1+3x_3-x_2}{3})=d(x_1)$$ for all $\mu \in \Bbb T^1$ and
all $x_1,x_2,x_3 \in A.$ By Lemma $2.2,$ the mapping
$d:A \to A$ is additive. Hence by Lemma $2.1,$ $d$ is $\Bbb C-$linear.\\
The rest of the proof is similar to the proofs of Theorem $2.4$
\end{proof}
\begin{cor}
Let $\theta$ be a nonnegative real number. Let $f:A \to A$ with
$f(0)=0$ be a mapping such that for some $\alpha,\beta,\gamma \in
\Bbb C$
$$\|f(\frac{x_2-x_1}{3})+f(\frac{x_1-3\mu x_3}{3})+\mu
f(\frac{3x_1+3x_3-x_2}{3})-f(x_1)\|_A \\
\leq \theta$$ and
$$\|\alpha f([x_1,x_2])-\beta[f(x_1),x_2]-\gamma[x_1,f(x_2)]\|_A\\
\leq \theta$$ for all $\mu \in \Bbb T^1$ and all $x_1,x_2,x_3 \in
A.$ Then there exists a unique $(\alpha,\beta,\gamma)-$derivation
$d:A \to A$ such that
$$\|d(x_1)-f(x_1)\|_A \leq \theta$$
for all $x_1 \in A.$
\end{cor} {\small
%----------------------------------------------------------------------%

}
\end{document}